\documentclass[10pt]{amsart}

\usepackage{graphicx}

\usepackage[noadjust]{cite}
\usepackage{color}
\usepackage[USenglish]{babel}
\usepackage{latexsym}
\usepackage{amsmath}
\usepackage{amsfonts}
\usepackage{amssymb}
\usepackage{esint}
\usepackage[all]{xy}
\renewcommand{\a }{\alpha }
\renewcommand{\b }{\beta }
\renewcommand{\d}{\delta }
\newcommand{\D }{\Delta }

\newcommand{\e }{\varepsilon }

\renewcommand{\l }{\lambda }

\newcommand{\n }{\nabla }

\newcommand{\s }{\sigma }

\renewcommand{\o }{\omega }
\renewcommand{\O }{\Omega }

\newcommand{\ov}{\overline}

\def\o{\omega}
\def\p{\partial}

\newcommand{\wtilde }{\widetilde}

\newcommand{\be}{\begin{equation}}
\newcommand{\ee}{\end{equation}}

\newcommand{\R}{\mathbb{R}}

\newcommand{\N}{\mathbb{N}}

\newcommand{\dis}{\displaystyle}

\newtheorem{theorem}{Theorem}[section]
\newtheorem{proposition}[theorem]{Proposition}
\newtheorem{definition}{Definition}[section]

\newtheorem{example}[theorem]{Example}

\newcommand{\bpr}{\begin{proposition}}
\newcommand{\epr}{\end{proposition}}
\newcommand{\bex}{\begin{example}\rm}
\newcommand{\eex}{\end{example}}

\def\aa{\`a }
\begin{document}

\newtheorem{lem}{Lemma}[section]
\newtheorem{pro}[lem]{Proposition}
\newtheorem{thm}[lem]{Theorem}
\newtheorem{rem}[lem]{Remark}
\newtheorem{cor}[lem]{Corollary}
\newtheorem{df}[lem]{Definition}

\title[Singular Mean Field Equations]
{Non-degeneracy and uniqueness of solutions to singular mean field equations on bounded domains}

\author{Daniele Bartolucci, Aleks Jevnikar, Chang-Shou Lin}

\address{Daniele Bartolucci, Department of Mathematics, University of Rome {\it "Tor Vergata"},  Via della Ricerca Scientifica 1, 00133 Roma, Italy.}
\email{bartoluc@mat.uniroma2.it}

\address{Aleks Jevnikar, Department of Mathematics, University of Pisa, Largo Bruno Pontecorvo 5, 56127 Pisa, Italy.}
\email{aleks.jevnikar@dm.unipi.it}

\address{Chang-Shou Lin, Taida Institute for Mathematical Sciences and Center for Advanced Study in Theoretical Sciences, National Taiwan University, Taipei, Taiwan.}
\email{cslin@math.ntu.edu.tw}

\thanks{D.B. and A.J. are partially supported  by PRIN project 2012, ERC PE1\_11, 
"{\em Variational and perturbative aspects in nonlinear differential problems}" and 
by the Consolidate the Foundations project 2015 "{\em Nonlinear Differential Problems and their Applications}" (sponsored by Univ. of Rome "Tor Vergata").\\
D.B. is partially supported by the Mission Sustainability project 2017 "SEEA" (sponsored by Univ. of Rome "Tor Vergata") and 
by "Fondo per le attivit\aa base di ricerca" MIUR 2017.}



\begin{abstract}
The aim of this paper is to complete the program initiated in
\cite{suz}, \cite{CCL} and then carried out by several authors concerning non-degeneracy and uniqueness of solutions
to mean field equations. In particular, we consider mean field equations with general singular data on non-smooth domains. 
The argument is based on the Alexandrov-Bol inequality and on the eigenvalues analysis of linearized singular Liouville-type problems. 
\end{abstract}

\maketitle

\section{Introduction}

We are concerned with the following Liouville-type problem 
\begin{equation} \label{eq0}
\left\{ \begin{array}{ll}
\D v+\rho \dfrac{e^v}{\int_{\O}e^v\,dx} = 4\pi\dis{\sum_{j=1}^N} \a_j\d_{p_j} & \mbox{in } \O, \vspace{0.2cm}\\
v=0 & \mbox{on } \p \O,
\end{array}
\right.
\end{equation}
and with its generalization in \eqref{eq2} below,
where $\O\subset\R^2$ is a simply-connected, open and bounded domain, $\rho$ is a positive parameter, 
$\{p_1,\dots,p_N\}\subset\O$ and $\a_j>-1$ for $j=1,\dots,N$. The latter equation arises as a mean field limit of 
turbulent Euler flows \cite{clmp1, clmp2, CK, Kies}. Its counterpart on manifolds is related to the Electroweak and Chern-Simons 
self-dual vortices \cite{sy2, T0, yang} and to the 
prescribed Gaussian curvature problem on surfaces \cite{Troy, KW, CY1, CY2}. 
Due to its relevance in mathematics and physics the literature for equation \eqref{eq} is huge and 
we just mention \cite{barjga, B5, BdM2, BdMM, bcct, BJLY2, BJLY, bl, BLin3, BLT, BMal, bt, bt2, bm, cl2, CLin4, cl4, dem2, DJLW, dj, GM1, GM2, GM3, yy, ls, Lin1, linwang, Mal1, 
PT, suz, Za2} and the references quoted therein.  

\medskip

To describe the main features of the problem we first write \eqref{eq0} as follows. Letting $G_p$, $p\in\O$, be the Green function,
\begin{equation} \label{green}
	\left\{ \begin{array}{rll}
						-\D G_p(y)=&\d_p & \mbox{in } \O, \\
						G_p(y)=&0 & \mbox{on } \p\O,
					\end{array}
					\right.
\end{equation} 
we say that $v$ is a solution of \eqref{eq0} if $u=v+4\pi\sum_{j=1}^N \a_j G_{p_j}(x)$ is an $H^1_0(\O)$ weak solution of,
\begin{equation} \label{eq}
\left\{ \begin{array}{ll}
\D u+\rho \dfrac{V(x)e^u}{\int_{\O}V(x)e^u\,dx} = 0 & \mbox{in } \O, \vspace{0.2cm}\\
u=0 & \mbox{on } \p \O,
\end{array}
\right.
\end{equation} 
where 
\begin{align} \label{V}
\begin{split}
	&V(x)=\mbox{exp}\Big( -4\pi\sum_{j=1}^N \a_j G_{p_j}(x) \Big), \\
	&V>0 \ \mbox{in } \O\setminus\{p_1,\dots,p_N\}, \quad V(x)\simeq|x-p_j|^{2\a_j} \ \mbox{near } p_j.
\end{split}
\end{align}
To avoid technicalities we postpone the discussion concerning the regularity assumptions on $\O$ to the sequel. 
Problem \eqref{eq} admits a variational formulation and, by a suitable adaptation of well known arguments \cite{moser, Troy}, 
the corresponding functional is seen to be coercive for $\rho<8\pi\bigr(1+\min_j\{\a_j,0\}\bigr)$. 
Therefore, in this range, weak solutions of \eqref{eq} are obtained by direct minimization. 
On the other hand, the non-degeneracy and uniqueness of solutions was first proved in \cite{suz}, 
where the author solves the regular case  ($N=0$) for $\rho<8\pi$ and $\O$ smooth and simply-connected.  
Later, this result was improved in \cite{CCL} to include the more delicate critical value $\rho=8\pi$, and finally generalized to 
the case of a possibly multiply-connected domain $\O$ in \cite{BLin3}.
The argument was also refined in \cite{bl} to cover the singular case where $\rho\leq 8\pi$ and $\a_j>0$ for all $j=1,\dots,N$, $N\geq 1$. 
Let us point out that the latter results are sharp in the sense that uniqueness does not hold in general if $\rho>8\pi$, see for example \cite{CCL, BdM2}. 
More recently, in \cite{wz} the authors considered the case of one negative singularity, i.e. $\a_1\in(-1,0)$ and at least one positive singularity 
$\a_j>0$ for all $j=2,\dots,N$, $N\geq2$, 
proving non-degeneracy and uniqueness of solutions provided $\rho\leq 8\pi(1+\a_1)$.

\medskip

There is a common strategy in the above mentioned results,  
which is based on rearrangement 
type arguments and the Alexandrov-Bol inequality, whose aim is to show that the first eigenvalue 
of the linearized problem of Liouville-type equations \eqref{eq} is strictly positive for $\rho\leq 8\pi\bigr(1+\min_j\{\a_j,0\}\bigr)$. 
This point is not trivial since the nonlinear term in \eqref{eq} is {\em constrained} in $L^1(\O)$, whence, roughly speaking, the associated first eigenvalue "looks like" an
higher order eigenvalue of a suitable unconstrained problem. 
As a matter of fact, for multiple negative sources (and even for a single negative source with no extra positive sources), the latter strategy is missing. 
Therefore, the first aim of this paper is to complete the above program and to prove non-degeneracy of solutions for \eqref{eq0} for 
general singular data on non-smooth domains. Actually, we will develop the strategy for a much more general problem, see \eqref{eq2} below, 
in which the sum of Dirac deltas may be replaced by a general measure of bounded variation. This is somehow the more general form of the singularities which 
one can attach to the Dirichlet problem for \eqref{eq} as it naturally arises in the analysis of Alexandrov surfaces with bounded integral curvature, 
see \cite{bc} and references therein.

\medskip

Once we have the non degeneracy, then the uniqueness of solutions to \eqref{eq},\eqref{eq2} will follow 
by the implicit function theorem and some uniform estimates for solutions to \eqref{eq2} 
with $\rho$ below the uniqueness threshold (see Theorem \ref{thm} below). At least to our knowledge these 
estimates are known only in the model case \eqref{eq}-\eqref{V}, as first derived in  \cite{bm}, \cite{ls} and then in \cite{bt} and 
\cite{BM3} via blow up analysis. However, it is not easy to generalize such a refined blow-up argument to the general singular weight in \eqref{eq2}. 
Also, the needed estimates were first derived in \cite{band}, but only in the analytic framework, and in \cite{bc2} for weak subsolutions, 
but only for the model case \eqref{eq}-\eqref{V}.
We solve this problem here by proving some uniform estimates of independent interest for weak solutions of \eqref{eq2}, in the same spirit of 
\cite{band,bc2}.\\
We finally remark that the uniqueness part
concerning solutions to \eqref{eq2} was very recently obtained in \cite{bgjm} by a completely different argument (see also \cite{GM3}
for a similar application of the latter method). From this point of view, we come up with a new proof of the uniqueness 
based on the non-degeneracy of \eqref{eq2}. 

\medskip

In order to introduce the problem let us fix the setting and some notations. As in \cite{CCL} we consider the following set of non-smooth domains.
\begin{definition}\label{def-dom}
We say that $\O$ is {\it regular} if its boundary is of class $C^2$ but
for a finite number of points $\{Q_1 , . . . , Q_{N_0}\}\subset \partial \O$ such that the following conditions holds at each $Q_j$:
\begin{itemize}
\item[(i)] The inner angle $\theta_j$ of $\partial \O$ at $Q_j$ satisfies $0 < {\theta_j \neq \pi} < 2\pi$;
\item[(ii)] At each $Q_j$ there is an univalent conformal map from $B_\delta (Q_j) \cap \overline{\O}$ to
the complex plane $\mathbb{C}$ such that $\partial \O \cap B_\delta (Q_j)$ is mapped to a $C^2$ curve.
\end{itemize}
\end{definition}
Obviously any non-degenerate polygon is regular according to this definition. Letting now $\O$ be regular, we will be interested in the following problem,
\begin{equation} \label{eq2}
\left\{ \begin{array}{ll}
\D u+\rho \dfrac{h(x)e^u}{\int_{\O}h(x)e^u\,dx} = 0 & \mbox{in } \O, \vspace{0.2cm}\\
u=0 & \mbox{on } \p \O.
\end{array}
\right.
\end{equation}
where $h=e^H$ is such that,
$$
	H=\mathcal H_+-\mathcal H_-,
$$
with $\mathcal H_+, \mathcal H_-$ two superharmonic functions defined by
\begin{equation} \label{om+}
	\mathcal H_{\pm}(x) = \mathfrak h_{\pm}(x) + \int_\O G_x(y)\,d\mu_{\pm}(y),
\end{equation}
where $\mathfrak h_{\pm}\in C^2(\O)\cap C^0(\ov\O)$ are harmonic functions in $\O$ and $\mu_{\pm}$ are non-negative and mutually orthogonal 
measures of bounded total variation \underline{compactly supported} in $\O$. 
\begin{definition} \label{def1}
Let $\o\subseteq\O$ be a nonempty subdomain. We denote by $\wtilde\o$ the interior of the closure of the union of $\o$ with its 
``holes", that is, with the bounded component of the complement of $\o$ in $\R^2$. 
\end{definition}

\begin{definition} 
Let $h=e^H$ and let $\mu_+$ be defined as in \eqref{om+}. Let $\o\subseteq\O$ be a nonempty subdomain and let $\wtilde\o$ be given as in Definition \ref{def1}. We define $\a(\o)=\a(\o,h)\geq 0$ to be
\begin{equation} \label{a}
	\a(\o)=\dfrac{1}{4\pi}\,\mu_+(\wtilde \o).
\end{equation}
\end{definition}

Moreover, we assume that 
\begin{equation}\label{hyp1}
 \a(\O)<1\;\;\mbox{i.e.}\;\; \mu_+(\O)<4\pi.
\end{equation}
\noindent
We remark that this is a rather natural condition, as it is somehow the minimal requirement needed to ensure that $h=e^{H}\in L^1(\O)$, whenever 
the measure $\mu_+$ is concentrated in just a single Dirac delta, which is the model problem \eqref{eq}-\eqref{V} with $N=1$. 

\begin{rem}\label{remxz}
By our assumptions on $h$, it is not difficult to see that $h$ is uniformly bounded from above and from below in a 
sufficiently small neighbourhood of $\p\O$. 
As a consequence, we can argue as in Lemma 2.1 in \cite{CCL} and prove that indeed $u\in C^0(\ov\O)$.\\ 
Next, since $\mu_+(\O)<4\pi$, 
then there exists at most one point $x_0\in\O$ such that $\mu_+(x_0)\geq 2\pi$. As a consequence, by arguing as in \cite{bc}, 
if $he^u\in L^1(\O)$ where $u\in L^1(\O)$ is a solution of \eqref{eq2} just in the sense of distributions, then we gain
$u\in W^{2,q}(\O)$ for some $q>1$ and in particular, for each $r>0$ small enough, there exists $s_r>2$ such that
$u\in W^{2,s_r}(\O\setminus B_r(x_0))$. 
We will refer to the latter property by saying that $u\in W^{2,s,\mbox{\rm \scriptsize loc}}(\O\setminus\{x_0\})$ for some $s>2$. 
Clearly $u$ is  a strong solution of \eqref{eq2} and similar integrability properties are deduced also on the weight $h$, see Proposition 1.4 in \cite{bc}.

\end{rem}

Our main result is the following.
\begin{thm}\label{thm}
Let $\O$ be an open, bounded, simply-connected and regular (according to Definition \ref{def-dom}) domain. 
Let $h$ be such that $\a=\a(\O,h)<1$, with $\a(\O,h)$ defined as in \eqref{a}. 
Then,  for any $\rho\leq8\pi(1-\a)$, there exists at most one weak $H^1_0(\O)$ solution of \eqref{eq2}
and the first eigenvalue of the corresponding linearized problem is strictly positive.
\end{thm}

\medskip

Observe that, by choosing $h(x)\equiv V(x)$, with $V$ given in \eqref{V}, and letting
$$
J=\bigr\{ j\in\{1,\dots,N\}\, : \, \a_j\in(-1,0) \bigr\},
$$
then $\a=\a(\O,h)>0$, as defined in \eqref{a}, is given by 
\begin{equation} \label{a1}
	\a=-\sum_{j\in J}\a_j.
\end{equation}
Then, Theorem \ref{thm} yields the following immediate Corollary for solutions of \eqref{eq}.
\begin{cor} \label{cor}
Let $\O$ be an open, bounded, simply-connected and regular (according to Definition \ref{def-dom}) domain. 
Let $\a<1$ be given as in \eqref{a1}. 
Then,  for any $\rho\leq8\pi(1-\a)$, there exists at most one weak $H^1_0(\O)$ solution of \eqref{eq} and 
the first eigenvalue of the corresponding linearized problem is strictly positive.
\end{cor}
As remarked above, the functional corresponding to \eqref{eq} is coercive for $\rho<8\pi\bigr(1+\min_j\{\a_j,0\}\bigr)$. Therefore, 
we have existence and uniqueness in Corollary \ref{cor}  if either $|J|\geq 2$ and $\rho\leq8\pi(1-\a)$ or if $|J|=1$ 
and $\rho<8\pi(1-\a)\equiv 8\pi(1+\a_1)$, $\a_1<0$, or if $|J|=0$ and $\rho<8\pi$.
Now, if $|J|=N=1$, $p_1=0$ and $\O=B_1(0)$, then solutions to \eqref{eq} (which are radial and well known in this particular case) exist if and only if 
$\rho<8\pi(1-\a)\equiv 8\pi(1+\a_1)$, showing that our existence and uniqueness result is sharp in this case. On the other side,
we stress that both our result and the one in \cite{bgjm} yield to the same uniqueness threshold which, for $|J|\geq 2$, is lower 
than the subcritical existence threshold $8\pi\bigr(1+\min_j\{\a_j,0\}\bigr)$. This motivates the following interesting open problem: 

\medskip

\textbf{Open problem.} Does uniqueness of solutions for \eqref{eq}-\eqref{V} hold 
for 
$$
\rho \in\bigr(8\pi(1-\a),8\pi\bigr(1+\min_j\{\a_j,0\}\bigr)\bigr), \quad \a>0,
$$ 
with $|J|\geq 2$?

\medskip

The strategy to prove Theorem \ref{thm} is inspired by the one in \cite{CCL}, 
with several non-trivial improvements needed to deal with general singular data (given by some measures as in \eqref{om+}) and non-smooth domains.
To this end, the first tool we need is an Alexandrov-Bol's inequality for solutions of \eqref{eq2} suitable for our setting. Such an inequality was first proved
in the analytical framework in \cite{band} and more recently generalized to the weak setting in \cite{bc2, bc}.
However what we need here is a more general statement which allows one to push the inequality, still in this weak setting, 
up to the (non-smooth) boundary of the domain. At least to our knowledge this is still missing and this is why we will derive it here.\\ 
Next, we perform a rearrangement argument jointly with a comparison method to 
gather some information about the eigenvalues of the linearized problem for \eqref{eq2}. Here, we present some novelties.\\  
On one side we generalize the argument in \cite{CCL} (introduced to treat the regular problem $N=0$ in the more subtle case $\rho=8\pi$) to 
the singular setting and on the other side we extend it to the sub-critical case $\rho<8\pi(1-\a)$. We point out that this step 
is new also for the regular problem as it simplifies the original argument in \cite{suz}. Special attention is also paid for the case 
$\rho=8\pi(1-\a)$ where one has to exploit the characterization of the equality in the Alexandrov-Bol inequality recently obtained
in \cite{bc} which, roughly speaking, asserts that equality can be attained only on simply-connected subdomains $\o\subseteq\O$ such 
that the full measure $\mu_+$ in \eqref{a} is concentrated in just one Dirac delta. Moreover, 
from the equality case in some Cauchy-Schwartz inequality, we will use the boundary regularity assumptions on the domain, 
see Definition \ref{def-dom}, in order to eventually apply the Hopf boundary lemma and get the desired conclusion. 
This step is done in the same spirit of \cite{CCL} and the boundary regularity assumptions are crucial at this point. 
This is in striking contrast with \cite{bl, wz}, where one can exploit the presence of positive singularities to readily conclude the argument 
even under weaker boundary regularity assumptions.

\medskip

This paper is organized as follows. In section \ref{sec-linearized} we introduce the Alexandrov-Bol inequality and analyze the  linearized 
Liouville-type problem. In section \ref{sec-proof} we prove some uniform estimates for solutions to \eqref{eq2} and then deduce Theorem \ref{thm}.

\medskip

\section{Eigenvalues analysis for Liouville-type linearized problems}  \label{sec-linearized}

In this section we first introduce the Alexandrov-Bol inequality suitable for our setting and then carry out an eigenvalues analysis for 
Liouville-type linearized problems which will be crucially used in the next section in the proof of Theorem~\ref{thm}. 

\medskip

\begin{definition} \label{def-simple}
We say that an open set $\O_0\subset\R^2$ is simple if $\p\O_0$ is a rectifiable Jordan curve whose interior is $\O_0$.
\end{definition}

Clearly any regular domain according to Definition \ref{def-dom} is also simple.
Next, given $\a\in[0,1)$, $\l>0$ we set
\begin{equation} \label{U}
	U_{\l,\a}(x) = \ln\left( \dfrac{ \l(1-\a) }{ 1+\frac{\l^2}{8}|x|^{2(1-\a)}} \right)^2,
\end{equation}
which satisfies
$$
	\D U_{\l,\a} + |x|^{-2\a}e^{U_{\l,\a}}=0 \mbox{ in }\R^2\setminus \{0\}.
$$
The following version of the Alexandrov-Bol inequality was first proved in the analytical framework in \cite{band} and more recently generalized to the weak setting in \cite{bc2, bc}. 
Actually, if $\omega$ in the statement is a relatively compact subset of $\O_0$, then the result is just a particular case of Theorem 1.5 in \cite{bc}. 
\begin{pro}\label{bol}
Let $\O_0\subset\R^2$ be a simple domain according to Definition \ref{def-simple}. Let $x_0\in\O_0$ be fixed as in Remark \ref{remxz} and let 
$w\in W^{2,s,\mbox{\rm \scriptsize loc}}(\O_0\setminus\{x_0\})\cap W^{2,q}(\O_0)\cap C^0(\ov\O_0)$ for some $s>2$ and some $q>1$, satisfy
$$
	\D w + h(x)e^{w}= 0  \mbox{ in } \O_0,
$$
where $h=e^H$ is such that $\a(\O_0,h)$ (as defined in \eqref{a}) satisfies $\a(\O_0,h)<1$.
Let $\omega\subseteq \O_0$ be any open subdomain such that $\p\o$ is a finite union of rectifiable Jordan curves and let $\a(\o)=\a(\o,h)$.  Then it holds,
\begin{equation} \label{bol-ineq}
	\left( \int_{\p\omega}\left(h(x)e^{w}\right)^{\frac 12}\,d\s \right)^2 \geq \frac 12 \left( \int_\omega h(x)e^w\,dx \right)\left( 8\pi(1-\a(\o))-\int_\omega h(x)e^w\,dx \right).
\end{equation}
Moreover, the equality holds if and only if (modulo conformal transformations)  
$\o=B_\d(0)$ for some $\d>0$, $h(x)e^{w}\equiv |x|^{-2\a}e^{U_{\l,\a}}$ for some $\l$ 
where $U_{\l,\a}$ is defined in \eqref{U}, $\mu_+=-\Delta H =4\pi \a \delta_{p=0}$ in $\o$ and $\a=\a(\o)$. In particular, if $\o$ is not simply-connected, then the inequality is always strict.
\end{pro} 

\begin{proof}
The proof can be  worked out by a step by step adaptation of the one provided in \cite{bc} with minor changes borrowed from  Theorem 4.1 in \cite{bc2}. 
Since $\omega\subseteq \O_0$ is such that $\p\o$ is a finite union of rectifiable Jordan curves and since $w\in C^0(\ov\o)$, 
then there exists $g\in C^0(\ov\o)$ such that $\D g=0$ in $\o$ and $g=w$ on $\p\o$. Set $\eta=w-g$,
$\eta\in W^{2,s,\mbox{\rm \scriptsize loc}}(\o\setminus\{x_0\})\cap W^{2,q}(\o)\cap C_0^0(\ov\o)$ for some $s>2$ and some $q>1$, which satisfies 
\begin{equation} \label{eta}
\left\{ \begin{array}{ll}
\D\eta +h(x)e^ge^\eta=0 & \mbox{in } \o, \vspace{0.2cm}\\
 \eta=0 & \mbox{on } \p \o.
\end{array}
\right.
\end{equation}
By using the strong maximum principle for weak solutions one can prove that 
$$
	\eta(x)>0 \ \forall x\in\o \quad \mbox{and} \quad \eta(x)=0 \ \mbox{iff} \ x\in\p\o.
$$
With this at hand the strategy follows the one introduced in \cite{bc}. 
Since we will need some ingredients later on we sketch here the main steps and refer to \cite{bc2, bc, bl} for full details. We first set,
$$
	\O(t)=\{x\in\o \,:\, \eta(x)>t\}, \quad \Gamma(t)=\p\O(t), \quad \mu(t)=\int_{\O(t)}h(x)e^g\,dx 
$$
and observe that $\O(0)=\o$. By the co-area formula one has for a.e. $t\geq0$,
\begin{equation} \label{co-area}
	\frac{d\mu(t)}{dt}=-\int_{\Gamma(t)} \frac{h(x)e^g}{|\n\eta|}\,d\s.
\end{equation}
Now, for $s\geq0$ we define the rearrangement $\eta^*$ of $\eta$,
\begin{equation} \label{eta*}
	\eta^*(s)=\bigr|\{ t\geq0 \,:\, \mu(t)>s \}\bigr|,
\end{equation}
where $|E|$ is the Lebesgue measure of a Borel set $E\subset\R$. It is not difficult to see that $\eta^*$ is the inverse of $\mu$. 
Moreover, in \cite{bc} it is shown that $\eta^*$ is locally Lipschitz. On the other hand, by \eqref{co-area} one has for a.e. $s\geq0$,
\begin{equation} \label{co-area2}
	\frac{d\eta^*(s)}{ds}=-\left(\int_{\Gamma(\eta^*(s))} \frac{h(x)e^g}{|\n\eta|}\,d\s\right)^{-1}.	
\end{equation} 
We next define for $s\geq0$,
$$
	F(s)=\int_{\O(\eta^*(s))} h(x)e^\eta e^g\,dx,
$$
and observe that $F(0)=0$ and $F(\mu(0))=M(\o)$, where we set
\begin{equation} \label{M}
	M(\o)=\int_\o h(x)e^w\,dx.
\end{equation}
By using the fact that $\eta^*$ is the inverse of $\mu$ and by \eqref{co-area} it is possible to check that
$$
	F(s)=\int_0^s e^{\eta^*(\l)}\,d\l
$$
and hence, for a.e. $s\geq0$,
\begin{equation} \label{F}
	F'(s)=e^{\eta^*(s)}, \quad F''(s)=\frac{d\eta^*(s)}{ds}\,e^{\eta^*(s)}=\frac{d\eta^*(s)}{ds}F'(s)\,.
\end{equation}
Now, we first use the Cauchy-Schwartz inequality, then \eqref{co-area2} and finally \eqref{eta} to get for a.e. $s\geq0$,
\begin{align*}
	\left( \int_{\Gamma(\eta^*(s))} (h(x)e^g)^{\frac12}\,d\s \right)^2 & \leq \left( \int_{\Gamma(\eta^*(s))} \frac{h(x)e^g}{|\n\eta|}\,d\s \right) \left(  \int_{\Gamma(\eta^*(s))} |\n\eta|\,d\s\right) \\
	& = \left(-\frac{d\eta^*(s)}{ds}\right)^{-1} \left(  \int_{\Gamma(\eta^*(s))} |\n\eta|\,d\s \right) \\
	& = \left(-\frac{d\eta^*(s)}{ds}\right)^{-1} \left(  \int_{\O(\eta^*(s))} h(x)e^\eta e^g\,dx \right) \\
	& = \left(-\frac{d\eta^*(s)}{ds}\right)^{-1} F(s).
\end{align*} 
On the other hand, the Huber inequality \cite{hu} asserts that for a.e. $s\geq0$,
$$
	\left( \int_{\Gamma(\eta^*(s))} (h(x)e^g)^{\frac12}\,d\s \right)^2 \geq 4\pi(1-\a(\o))\mu(\eta^*(s))=4\pi(1-\a(\o))s.
$$
It follows that for a.e. $s\geq0$,
\begin{equation} \label{est-F}
	4\pi(1-\a(\o))s \leq \left(-\frac{d\eta^*(s)}{ds}\right)^{-1} F(s).
\end{equation}
Letting,
$$
	P(s)=4\pi(1-\a(\o))\left(s F'(s)-F(s)\right)+\frac12 F^2(s),
$$
we deduce from \eqref{F} and \eqref{est-F} that,
\begin{equation} \label{P}
	\frac{d}{ds}P(s)\geq0 \quad \mbox{for a.e. } s\geq0. 
\end{equation}
Since the functions involved in the definition of $P$ are locally Lipschitz continuous, after integration we end up with,
$$
	P(\mu(0))-P(0)\geq 0,
$$	
which is equivalent to,
$$
	8\pi(1-\a(\o))(\mu(0)-M(\o))+M(\o)^2 \geq0,
$$
where $M(\o)$ is defined in \eqref{M}. Using Huber's inequality once more we deduce that,
\begin{align*}
	\left( \int_{\p\omega}\left(h(x)e^{w}\right)^{\frac 12}\,d\s \right)^2 & = \left( \int_{\p\omega}\left(h(x)e^{g}\right)^{\frac 12}\,d\s \right)^2 \\
		& \geq 4\pi(1-\a(\o))\mu(0) \\
		& \geq \frac12 M(\o)\bigr(8\pi(1-\a(\o))-M(\o)\bigr),
\end{align*}
which is the desired inequality in \eqref{bol-ineq}. The characterization	of the equality case can be carried out as in \cite{bc} and we skip the details.
\end{proof}

\medskip

Next we consider the eigenvalue problem for a linearized Liouville-type equation, by recalling that a nodal domain for $\phi\in C^0(\ov\O)$ 
is any maximal connected component of the subdomain where $\phi$ has a definite sign. 
We will need a Gauss-Green formula and a Courant nodal line Theorem suitable to be applied in our weak setting.
Even under our weak summability assumptions about $h$, still these results are well known, see for example \cite{kok}. 
Therefore we omit the proof of the following Lemma which, in view of our assumptions on $h$, can be obtained by a rather standard adaptation 
of the one worked out in \cite{bl}.
\begin{lem} \label{lem-nodal}
Let $\O\subset\R^2$ be an open, bounded, simply-connected and piecewise $C^2$ domain according to Definition \ref{def-dom}. Let
$w-c\in H^1_0(\O)$ for some $c\in\R$ and $w$ be a weak solution of, 
$$
	\D w + h(x)e^{w}= 0  \mbox{ in } \O,
$$
where $h=e^H$ is such that $\a=\a(\O,h)$ (as defined in \eqref{a}) satisfies $\a(\O,h)<1$.

\medskip

Suppose that for some $\hat\nu$, either $\phi\in H_0^1(\O)$ is a weak solution of, 
\begin{equation} \label{lin1}
\left\{ \begin{array}{ll}
-\D\phi -h(x)e^w\phi=\hat\nu h(x)e^w\phi & \mbox{in } \O, \vspace{0.2cm}\\
 \phi=0 & \mbox{on } \p \O,
\end{array}
\right.
\end{equation}
or that exists $c_0\in\R$, $c_0<0$ such that $\phi-c_0\in H_0^1(\O)$ and $\phi$ is a weak solution of
\begin{equation} \label{lin2}
\left\{ \begin{array}{ll}
-\D\phi -h(x)e^w\phi=\hat\nu h(x)e^w\phi & \mbox{in } \O, \vspace{0.2cm}\\
 \phi=c_0 & \mbox{on } \p \O, \vspace{0.2cm}\\
\int_\O h(x)e^w\phi\,dx=0.
\end{array}
\right.
\end{equation}
Then, for a nodal domain $\o\subseteq\O$ for $\phi$, it holds,
{
\begin{equation} \label{gg1}
	-\int_\o \phi\D\phi\,dx=\int_\o|\n\phi|^2\,dx,
\end{equation}
and
\begin{equation} \label{gg2}
	\int_\o|\n\phi|^2\,dx=(\hat\nu+1)\int_\o h(x)e^w|\phi|^2\,dx.
\end{equation} 
}
Moreover, let
$$
	\bigr(\hat\nu_k, \phi_k^{(j)}\bigr) \quad k\in\N, \quad j=1,\dots,l_k, \, l_k\in\N, \quad \hat\nu_1<\hat\nu_2<\dots,
$$
be the eigenvalues and the corresponding eigenfunctions for \eqref{lin1}. 
Then, $\hat\nu_1$ is simple ($l_1=1$), $\phi_1$ has only one nodal domain { and the second eigenfunction has exactly two nodal domains.
Finally, any other eigenfunction has at least two nodal domains.}
\end{lem}

\medskip

With this at hand we can start the eigenvalues analysis of $(-\D-h(x)e^w)(\cdot)$. 
In particular, by assuming some bounds on $\int_\O h(x)e^w\,dx$ we will derive useful information on the first and second eigenvalue. 
\begin{pro} \label{pro-nodal}
Let $\O$, $h$ and $w$ be as in Lemma \ref{lem-nodal}. Let $\hat\nu_1, \hat\nu_2$ be the first and second eigenvalues for \eqref{lin1}, respectively. Then, it holds:

\smallskip

$(i)$ If $\int_\O h(x)e^w\,dx< 4\pi(1-\a)$, then $\hat\nu_1>0$.

\smallskip

$(ii)$ If $\int_\O h(x)e^w\,dx\leq 8\pi(1-\a)$, then $\hat\nu_2>0$.
\end{pro}

\begin{proof}
First of all, by Remark \ref{remxz} we have 
$w\in W^{2,s,\mbox{\rm \scriptsize loc}}(\O\setminus\{x_0\})\cap W^{2,q}(\O)\cap C^0(\ov\O)$ for some $s>2$, $q>1$ and 
some $x_0\in\O$. Similar regularity properties are deduced on the eigenfunctions $\phi$ of \eqref{lin1}. 
Clearly, an eigenvalue and eigenfunction $(\hat\nu,\phi)$ for \eqref{lin1} correspond to an eigenvalue and eigenfunction $(\nu=\hat\nu+1,\phi)$ for,
\begin{equation} \label{lin3}
\left\{ \begin{array}{ll}
-\D\phi =\nu h(x)e^w\phi & \mbox{in } \O, \vspace{0.2cm}\\
 \phi=0 & \mbox{on } \p \O.
\end{array}
\right.
\end{equation}

\medskip

\textbf{Proof of $(i)$.} 
Suppose $\int_\O h(x)e^w\,dx< 4\pi(1-\a)$ and suppose by contradiction that there exists $\phi=\phi_1$, 
the first eigenfunction for \eqref{lin1}, with $\hat\nu_1\leq 0$. Then, we have
$$
\left\{ \begin{array}{ll}
-\D\phi =\nu_1 h(x)e^w\phi & \mbox{in } \O, \vspace{0.2cm}\\
 \phi=0 & \mbox{on } \p \O,
\end{array}
\right.
$$
with $\nu_1\leq 1$. Moreover, by Lemma \ref{lem-nodal} we know that $\phi$ has only one nodal domain and w.l.o.g. we assume $\phi\geq0$ in $\O$. In particular, by the maximum principle for weak solutions we have $\phi>0$ in $\O$. Recalling \eqref{U} we set $U_\a(x):=U_{1,\a}(x)$, i.e.
\begin{equation} \label{Ua}
	U_{\a}(x) = \ln\left( \dfrac{ (1-\a) }{ 1+\frac{1}{8}|x|^{2(1-\a)}} \right)^2,
\end{equation}
which satisfies,
$$
	\D U_{\a} + |x|^{-2\a}e^{U_{\a}}=0 \mbox{ in }\R^2\setminus \{0\}.
$$
Next, let $t_+=\max_{\ov\O}\phi$ and for $t>0$ we define $\O_t=\{x\in\O \,:\,\phi>t\}$ and $R(t)>0$ such that
$$
	\int_{B_R(t)}|x|^{-2\a}e^{U_{\a}}\,dx=\int_{\O_t} h(x)e^w\,dx.
$$
Since $\phi>0$ in $\O$ we put $\O_0=\O$. Moreover, $R_0=\lim_{t\to0^+}R(t)$ and $\lim_{t\to (t_+)^-}R(t)=0$. Then $\phi^*:B_{R_0}\to\R$, which for $y\in B_{R_0}$, $|y|=r$, is given by,
$$
	\phi^*(r)=\sup\{t\in(0,t_+)\,:\,R(t)>r\},
$$ 
is a radial, decreasing, equimeasurable rearrangement of $\phi$ with respect to the measures $h(x)e^w\,dx$ and $|x|^{-2\a}e^{U_{\a}}$, and hence, in particular,
\begin{align}
	&B_{R(t)}=\{x\in\R^2\,:\,\phi^*(x)>t\}, \nonumber\\
	&\int_{\{\phi^*>t\}}|x|^{-2\a}e^{U_{\a}}\,dx=\int_{\O_t} h(x)e^w\,dx \quad t\in[0,t_+), \nonumber\\
	&\int_{B_{R_0}}|x|^{-2\a}e^{U_{\a}}|\phi^*|^2\,dx=\int_{\O} h(x)e^w|\phi|^2\,dx. \label{rearr}
\end{align}
Clearly, $\phi^*$ is a BV function. We apply the Cauchy-Schwartz inequality and the co-area formula to get that, 
\begin{align}
	\int_{\{\phi=t\}} |\n \phi|\,d\s & \geq \left( \int_{\{\phi=t\}} \left(h(x)e^{w}\right)^{\frac 12}\,d\s \right)^2 \left( \int_{\{\phi=t\}} \dfrac{h(x)e^u}{|\n \phi|}\,d\s \right)^{-1} \label{c-s}\\
	  & = \left( \int_{\{\phi=t\}} \left(h(x)e^{w}\right)^{\frac 12}\,d\s \right)^2 \left( -\dfrac{d}{dt}\int_{\O_t} h(x)e^w\,dx \right)^{-1}, \nonumber
\end{align}
for a.e. $t$. Then, by means of the Alexandrov-Bol inequality in Proposition \ref{bol} we have,
\begin{align*}
	&\left( \int_{\{\phi=t\}} \left(h(x)e^{w}\right)^{\frac 12}\,d\s \right)^2 \left( -\dfrac{d}{dt}\int_{\O_t} h(x)e^w\,dx \right)^{-1} \\
	&\geq \frac 12 \left( \int_{\O_t} h(x)e^w\,dx \right)\left( 8\pi(1-\a)-\int_{\O_t} h(x)e^w\,dx \right) \left( -\dfrac{d}{dt}\int_{\O_t} h(x)e^w\,dx \right)^{-1}.
\end{align*}
Since $\phi^*$ is an equimeasurable rearrangement of $\phi$ with respect to the measures $h(x)e^u \,dx$, $|x|^{-2\a}e^{U_{\a}}\,dx$, and since 
$|x|^{-2\a}e^{U_{\a}}$ realizes the equality in Proposition~\ref{bol}, we get,
\begin{align*}
&\frac 12 \left( \int_{\O_t} h(x)e^w\,dx \right)\left( 8\pi(1-\a)-\int_{\O_t} h(x)e^w\,dx \right) \left( -\dfrac{d}{dt}\int_{\O_t} h(x)e^w\,dx \right)^{-1} \\
&=\frac 12 \left( \int_{\{\phi^*>t\}} |x|^{-2\a}e^{U_{\a}}\,dx \right)\left( 8\pi(1-\a)-\int_{\{\phi^*>t\}} |x|^{-2\a}e^{U_{\a}}\,dx \right) \left( -\dfrac{d}{dt}\int_{\{\phi^*>t\}} |x|^{-2\a}e^{U_{\a}}\,dx \right)^{-1} \\
& =\left( \int_{\{\phi^*=t\}} \left(|x|^{-2\a}e^{U_{\a}}\right)^{\frac 12}\,d\s \right)^2 \left( -\dfrac{d}{dt}\int_{\{\phi^*>t\}} |x|^{-2\a}e^{U_{\a}}\,dx \right)^{-1} \\
& = \int_{\{\phi^*=t\}} |\n \phi^*|\,d\s,
\end{align*}
where in the last equality we used the co-area formula for BV functions, see \cite{fr}. Therefore, we have proved that,
$$
\int_{\{\phi^*=t\}} |\n \phi^*|\,d\s \leq\int_{\{\phi=t\}} |\n \phi|\,d\s,
$$
for a.e. $t$, which in turn yields
\begin{equation} \label{grad}
	\int_{B_{R_0}} |\n \phi^*|^2\,dx \leq \int_\O |\n \phi|^2\,dx.
\end{equation}

With the latter estimate at hand we can use \eqref{rearr} and \eqref{gg2} (or the variational characterization of $\phi$) to deduce that,
$$
	\int_{B_{R_0}} |\n \phi^*|^2\,dx-\int_{B_{R_0}}|x|^{-2\a}e^{U_{\a}}|\phi^*|^2\,dx \leq \int_\O |\n \phi|^2\,dx-\int_{\O} h(x)e^w|\phi|^2\,dx 
$$
$$	
 =(\nu_1-1)\int_{\O} h(x)e^w|\phi|^2\,dx\leq0,
$$
since $\nu_1\leq1$ by assumption. Moreover, $\phi^*(R_0)=0$. Therefore, we conclude that the first eigenvalue for $(-\D-h(x)e^w)(\cdot)$ on $B_{R_0}$ with Dirichlet boundary conditions is non-positive. Consider now $\psi(x)=\frac{8-|x|^{2(1-\a)}}{8+|x|^{2(1-\a)}}$ which satisfies
$$
	-\D\psi-|x|^{-2\a}e^{U_{\a}}\psi=0 \mbox{ in } \R^2\setminus\{0\}, \quad \psi\in H^1_0(B_{R}(0)), \, R=8^{\frac{1}{2(1-\a)}}.
$$
Since the first eigenvalue is non-positive one can deduce that $R_0\geq8^{\frac{1}{2(1-\a)}}$. Moreover, 
\begin{equation} \label{R0}
	8\pi(1-\a)\frac{R_0^{2(1-\a)}}{8+R_0^{2(1-\a)}}=\int_{B_{R_0}}|x|^{-2\a}e^{U_{\a}}\,dx=\int_\O h(x)e^w\,dx<4\pi(1-\a),
\end{equation}
by assumption and hence $R_0<8^{\frac{1}{2(1-\a)}}$, yielding a contradiction. 

\

\textbf{Proof of $(ii)$.} Suppose now $\int_\O h(x)e^w\,dx\leq 8\pi(1-\a)$ and suppose by contradiction that there exists $\phi=\phi_2$, 
a second eigenfunction for \eqref{lin3} corresponding to a second eigenvalue $\nu_2$ with $\nu_2=\hat\nu_2+1\leq1$. 
{ From Lemma \ref{lem-nodal} we know that $\phi$ has exactly two nodal domains:}
\begin{equation} \label{pm}
	\O_+=\{x\in\O \,:\,\phi(x)>0\}, \quad \O_-=\{x\in\O \,:\,\phi(x)<0\}.
\end{equation}
Suppose first that,
$$
	\int_{\O_+} h(x)e^w\,dx<4\pi(1-\a).
$$
In this case we exploit the rearrangement argument introduced in the proof of $(i)$ and just replace $\O$ with $\O_+$. { By using also \eqref{gg2} we end up with}
\begin{equation} \label{rearr-ineq}
	\int_{B_{R_0}} |\n \phi^*|^2\,dx-\int_{B_{R_0}}|x|^{-2\a}e^{U_{\a}}|\phi^*|^2\,dx \leq \int_{\O_+} |\n \phi|^2\,dx-\int_{\O_+} h(x)e^w|\phi|^2\,dx 
\end{equation}
$$	
 =(\nu_2-1)\int_{\O_+} h(x)e^w|\phi|^2\,dx\leq0,
$$
and then the same argument used in the proof of $(i)$ yields to a contradiction.

\medskip

If instead $\int_{\O_+} h(x)e^w\,dx>4\pi(1-\a)$ we may switch the role of $\O_+$ and $\O_-$ and apply again the above argument. Therefore, we are left with the case 
$$
	\int_{\O_+} h(x)e^w\,dx=\int_{\O_-} h(x)e^w\,dx=4\pi(1-\a), \quad \int_{\O} h(x)e^w\,dx=8\pi(1-\a).
$$
In this case we necessarily have $R_0=8^{\frac{1}{2(1-\a)}}$ and the first eigenvalue for $(-\D-|x|^{-2\a}e^{U_{\a}})(\cdot)$ on $B_{R_0}$ is zero. Therefore, 
$$
	\int_{B_{R_0}} |\n \phi^*|^2\,dx-\int_{B_{R_0}}|x|^{-2\a}e^{U_{\a}}|\phi^*|^2\,dx=0,
$$
and hence, in particular, the inequality in \eqref{rearr-ineq} turns out to be an equality. 
This yields to the equality in \eqref{grad} as well. On one side, the latter equality holds if and only if we have equality in the 
Alexandrov-Bol inequality \eqref{bol-ineq} and hence, in particular, $\O_+$ is simply-connected and 
$h(x)e^{w}$ is such that $\mu_+=-\D H=4\pi\a\d_{p}$ in $\O_+$. The same holds for $\O_-$. 
At this point, if we were in presence of positive singular sources, then the latter facts would force the positive singular sources to 
be supported on the nodal line of $\phi$ in $\O$ thus contradicting a result of \cite{bers}  (see \cite{bl} for details) and the conclusion would follow.  
The general case is more delicate. 

\medskip

The equality in the Cauchy-Schwartz inequality \eqref{c-s} implies, for a.e. $t$,
\begin{equation} \label{c-s2}
	h(x)e^{w(x)}=c_t|\n\phi(x)|^2,
\end{equation}
for some constant $c_t$ depending on $t$, for all $x$ such that $\phi(x)=t$. 
Since both $\O_+$ and $\O_-$ are simply-connected, the nodal line, which is the closure of $\{\phi(x)=0\,:\,x\in\O\}$, 
must intersect $\p\O$ and moreover, we may assume without loss of generality that $\p\O_+\cap\p\O$ contains an arc of positive length. 
Now, since $w=0$ on $\p\O$ and since $h$ is uniformly bounded from above and from below in a sufficiently small neighborhood of $\p\O$, 
see Remark \ref{remxz}, by letting $t\to0$ we deduce from \eqref{c-s2} that $\frac{|\n\phi|}{h}$ is constant on any $C^2$ portion of $\p\O_+\cap\p\O$.\\ 
At this point, let $x_0$ be a point of the intersection of the nodal line with $\p\O$. If $x_0$ is a smooth point of the boundary we readily have $\phi\in C^{1,\b}$, $\b\in(0,1)$ at $x_0$ and necessarily $|\n\phi(x_0)|=0$. Since $\frac{|\n\phi|}{h}$ is constant, it follows that $|\n\phi|=0$ on the $C^2$ portion of $\p\O_+\cap\p\O$ containing $x_0$. This is in contradiction to the Hopf boundary point lemma.  If $x_0$ is not a smooth point of the boundary we can proceed as in Case 2 of Lemma 4.3 in \cite{CCL} and exploit the properties on $\p\O$ as given in Definition~\ref{def-dom} to get an analogous contradiction to the Hopf lemma. 
\end{proof}

\medskip

We consider now the eigenvalue problem with non-null boundary conditions \eqref{lin2} and follow an argument in \cite{CCL}. 
In particular, we generalize the proof in \cite{CCL} to the singular case and extend it to the sub-critical regime $\rho<8\pi(1-\a)$. 
This step is new also for the regular case and simplifies the original argument due to \cite{suz}. We have the following property.

\begin{pro} \label{pro-nodal2}
Let $\O$, $h$ and $w$ be as in Lemma \ref{lem-nodal}. Let $\hat\nu$ be an eigenvalue for \eqref{lin2}. If $\int_{\O}h(x)e^w\,dx\leq8\pi(1-\a)$, then $\hat\nu>0$. 
\end{pro}

\begin{proof}
Suppose by contradiction that there exists an eigenfunction $\phi$ such that,
$$
\left\{ \begin{array}{ll}
-\D\phi =\nu h(x)e^w\phi & \mbox{in } \O, \vspace{0.2cm}\\
 \phi=c_0 & \mbox{on } \p \O, \vspace{0.2cm}\\
\int_\O h(x)e^w\phi\,dx=0,
\end{array}
\right.
$$
for some $c_0<0$ and some $\nu\leq1$. Let $\O_+$ and $\O_-$ be defined as in \eqref{pm}. We start by showing that
\begin{equation} \label{o+}
	\int_{\O_+} h(x)e^w\,dx \geq 4\pi(1-\a).
\end{equation}
First observe that clearly $\O_+\subset\subset\O$. The estimate in \eqref{o+} will follow by showing that 
$$
	\int_{\o_+} h(x)e^w\,dx \geq 4\pi(1-\a),
$$
for any connected component $\o_+$ of $\O_+$. Indeed, suppose this is not the case. Then, if $\o_+$ is simply-connected, $\phi$ is the first eigenfunction for \eqref{lin1} on $\o_+$ and Proposition~\ref{pro-nodal} $(i)$ implies $\hat\nu=\hat\nu_1>0$, a contradiction. If instead $\o_+$ is multiply connected it is not difficult to see that we can apply Proposition \ref{pro-nodal} $(ii)$ to get $\hat\nu=\hat\nu_2>0$, a contradiction again. This completes the proof of \eqref{o+}.

\medskip

Letting now
$$
	\wtilde\O_+=\{x\in\O \,:\,\phi(x)>c_0\}, \quad \wtilde\O_-=\{x\in\O \,:\,\phi(x)< c_0\},
$$
we distinguish two cases.

\medskip

\textbf{Case 1.} Suppose that $\wtilde\O_-=\emptyset$. In this case we perform the same rearrangement argument for $\phi$ on $\O$ introduced in Proposition \ref{pro-nodal} $(i)$. For any $t>c_0$ let $\phi^*$ be the radial, decreasing, equimeasurable rearrangement of $\phi$ with respect to the measures $h(x)e^w\,dx$ and $|x|^{-2\a}e^{U_{\a}}$, where $U_\a$ is given in \eqref{Ua}. In particular,
\begin{equation} \label{tot}
	\int_{B_{R_0}}|x|^{-2\a}e^{U_{\a}}\,dx=\int_{\O} h(x)e^w\,dx,
\end{equation}
where $R_0=+\infty$ if $\int_{\O} h(x)e^w\,dx=8\pi(1-\a)$. As in Proposition \ref{pro-nodal} we have,
\begin{align*}
	&\int_{B_{R_0}}|x|^{-2\a}e^{U_{\a}}|\phi^*|^2\,dx=\int_{\O} h(x)e^w|\phi|^2\,dx, \\
	&\int_{B_{R_0}}|x|^{-2\a}e^{U_{\a}}\phi^*\,dx=\int_{\O} h(x)e^w\phi\,dx=0, \\
	&\int_{B_{R_0}} |\n \phi^*|^2\,dx \leq \int_\O |\n \phi|^2\,dx.
\end{align*}
Moreover,
$$
	\int_{B_{R_0}} |\n \phi^*|^2\,dx-\int_{B_{R_0}}|x|^{-2\a}e^{U_{\a}}|\phi^*|^2\,dx \leq \int_\O |\n \phi|^2\,dx-\int_{\O} h(x)e^w|\phi|^2\,dx 
$$
\begin{equation} \label{rearr-ineq2}
 =(\nu-1)\int_{\O} h(x)e^w|\phi|^2\,dx\leq0.
\end{equation}
We then define,
\begin{align} \label{K}
\begin{split}
	K^*=\inf\left\{ \int_{B_{R_0}} |\n \psi|^2\,dx \right. & : \psi\in H_{\mbox{\rm \scriptsize rad}}(B_{R_0}), \ \int_{B_{R_0}}|x|^{-2\a}e^{U_{\a}}\psi\,dx=0, \\
			& \ \left. \int_{B_{R_0}}|x|^{-2\a}e^{U_{\a}}|\psi|^2\,dx=1 \right\},
\end{split}
\end{align}
where $H_{\mbox{\rm \scriptsize rad}}(B_{R_0})$ stands for radial functions $\psi$ with $\psi\in L^2(B_{R_0},|x|^{-2\a}e^{U_{\a}}\,dx)$, $|\n\psi|\in L^2(B_{R_0})$. We point out once more that $B_{R_0}=\R^2$ in case $\int_{\O} h(x)e^w\,dx=8\pi(1-\a)$. Observe that by construction and by the property \eqref{rearr-ineq2} of $\phi^*$, we obtain $K^*\leq 1$ for any $\rho\leq8\pi(1-\a)$. We distinguish now between two cases.

\medskip

Suppose first $\rho<8\pi(1-\a)$ and $R_0<+\infty$. We will extend here the argument introduced in \cite{CCL}. 
In particular, this simplifies the original argument due to \cite{suz}. 
We start by observing that, as in Proposition \ref{pro-nodal}, in analogy with \eqref{R0}, 
the fact that $\int_\O h(x)e^w\,dx>\int_{\O_+} h(x)e^w\,dx\geq 4\pi(1-\a)$, see \eqref{o+}, implies $R_0>8^{\frac{1}{2(1-\a)}}$. Let now $\psi^*$ be the minimizer of \eqref{K} which satisfies
$$
	\D\psi^* +K^*|x|^{-2\a}e^{U_{\a}}\psi^*=0 \quad \mbox{in } B_{R_0}\setminus\{0\}, 
$$
and
\begin{equation} \label{zero}
	\int_{B_{R_0}}|x|^{-2\a}e^{U_{\a}}\psi^*\,dx=0.
\end{equation}
By the latter property $\psi^*$ changes sign in $B_{R_0}$. On the other hand, we already know that $K^*\leq 1$ and hence $\psi^*$ changes sign 
only once otherwise we may use Proposition \ref{pro-nodal} $(ii)$ to get a contradiction. Therefore, we may assume that there exists $\xi_0\in(0,R_0)$ such that,
$$
\left\{ \begin{array}{ll}
	\psi^*(r)>0 & \mbox{for } r\in[0,\xi_0), \\
	\psi^*(\xi_0)=0,	& \\
  \psi^*(r)<0 & \mbox{for } r\in(\xi_0,R_0].
\end{array}
\right.
$$
By integrating the equation for $\psi^*$ and by using \eqref{zero} we deduce,
$$
	R_0(\psi^*)'(R_0)= -K^*\int_0^{R_0} |s|^{-2\a}e^{U_{\a}(s)}\psi^*(s)s\,ds=0.
$$
 Therefore, so far we can assert that,
\begin{equation} \label{psi*}
\psi^*(R_0)<0, \quad (\psi^*)'(R_0)=0.
\end{equation}
Consider now $\psi(x)=\frac{8-|x|^{2(1-\a)}}{8+|x|^{2(1-\a)}}$ which satisfies
$$
\left\{ \begin{array}{ll}
	\D\psi+|x|^{-2\a}e^{U_{\a}}\psi=0 & \mbox{in } \R^2\setminus\{0\}, \vspace{0.2cm}\\
  \psi(x)=0 & \mbox{for } |x|=8^{\frac{1}{2(1-\a)}}.
\end{array}
\right.
$$
In particular, since $R_0>8^{\frac{1}{2(1-\a)}}$ we have,
\begin{equation} \label{psi}
	\psi(R_0)<0, \quad \psi'(R_0)<0.
\end{equation}
We aim to show that $\xi_0=8^{\frac{1}{2(1-\a)}}$. Similarly as before, by using the equations for $\psi$ and $\psi^*$ it is not difficult to check that,
\begin{equation} \label{sign}
	r\left( \frac{\psi^*}{\psi} \right)'\!(r)\psi^2(r)= (1-K^*)\int_0^{r} |s|^{-2\a}e^{U_{\a}(s)}\psi^*(s)\psi(s)s\,ds.
\end{equation}
Suppose by contradiction $\xi_0<8^{\frac{1}{2(1-\a)}}$. Then, since $K^*\leq1$, the right-hand side in \eqref{sign} is non-negative for $r\leq\xi_0$ and hence $\frac{\psi^*(r)}{\psi(r)}$ is non-decreasing for $r\leq\xi_0$. It follows that, 
$$
	0<\frac{\psi^*(0)}{\psi(0)}\leq \frac{\psi^*(\xi_0)}{\psi(\xi_0)}=0,
$$
which is a contradiction.  Now, analogous computations as before yield,
$$
	R_0\left( \frac{\psi^*}{\psi} \right)'\!(R_0)\psi^2(R_0)-r\left( \frac{\psi^*}{\psi} \right)'\!(r)\psi^2(r)= (1-K^*)\int_r^{R_0} |s|^{-2\a}e^{U_{\a}(s)}\psi^*(s)\psi(s)s\,ds.
$$
Observe that by \eqref{psi*} and \eqref{psi} we have,
\begin{equation} \label{c1}
	(\psi^*)'(R_0)\psi(R_0)-\psi'(R_0)\psi^*(R_0)<0.
\end{equation}
It follows that,
\begin{equation} \label{sign2}
	-r\left( \frac{\psi^*}{\psi} \right)'\!(r)\psi^2(r)= (1-K^*)\int_r^{R_0} |s|^{-2\a}e^{U_{\a}(s)}\psi^*(s)\psi(s)s\,ds+C_0,
\end{equation}
for some $C_0>0$. Suppose now by contradiction $\xi_0>8^{\frac{1}{2(1-\a)}}$. Then, the righ-hand side in \eqref{sign2} is positive for $r\geq\xi_0$ and hence $\frac{\psi^*(r)}{\psi(r)}$ is decreasing for $r\geq\xi_0$. Therefore, recalling \eqref{psi*} and \eqref{psi},
$$
	0=\frac{\psi^*(\xi_0)}{\psi(\xi_0)}>\frac{\psi^*(R_0)}{\psi(R_0)}>0,
$$
which is a contradiction again. We conclude that necessarily $\xi_0=8^{\frac{1}{2(1-\a)}}$ and, in particular, 
\begin{equation} \label{sign3}
	\psi^*(r)\psi(r)>0 \quad \mbox{for all } r\neq8^{\frac{1}{2(1-\a)}}.
\end{equation}	
Finally, by using once more \eqref{sign2} with $r=0$ we have,
\begin{equation} \label{c2}
	(1-K^*)\int_0^{R_0} |s|^{-2\a}e^{U_{\a}(s)}\psi^*(s)\psi(s)s\,ds=-C_0<0.
\end{equation}
Since \eqref{sign3} holds true, then the latter inequality implies $K^*>1$ which is a contradiction. This concludes the proof of Case 1 for $\rho<8\pi(1-\a)$.

\medskip

Consider now the case $\rho=8\pi(1-\a)$ and $B_{R_0}=\R^2$.
The same argument adopted above in this situation shows that 
the quantity in \eqref{c1} vanishes as $R_0\to+\infty$. 
Then we have $C_0=0$ in \eqref{sign2} and \eqref{c2} which in turn imply that $K^*=1$. We skip the details to avoid repetitions and 
refer to \cite{CCL} for more details concerning this point. Therefore, the inequality in \eqref{rearr-ineq2} turns out to be an equality. 
This yields to the equality in \eqref{grad} as well and, in particular, to the equality in the Alexandrov-Bol inequality \eqref{bol-ineq}. In particular,
$$
\left( \int_{\{\phi=t\}} \left(h(x)e^{w}\right)^{\frac 12}\,d\s \right)^2 = \frac 12 \left( \int_{\O_t} h(x)e^w\,dx \right)\left( 8\pi(1-\a)-\int_{\O_t} h(x)e^w\,dx \right),
$$
for a.e. $t>c_0$. However, the left-hand side of the latter equality is uniformly strictly positive, while the right-hand side tends to zero for $t\to c_0^+$. 
This yields to a contradiction in the case $\rho=8\pi(1-\a)$ as well.

\medskip

\textbf{Case 2.} Suppose now that $\wtilde\O_-\neq\emptyset$. In this case we start by considering the rearrangement $\phi^*$ of $\phi$ in $\wtilde\O_+$. In particular, 
$$
	\int_{B_{R_0^+}}|x|^{-2\a}e^{U_{\a}}\,dx=\int_{\wtilde\O_+} h(x)e^w\,dx,
$$
for some $R_0^+$, while $R_0$ is given by \eqref{tot} where we recall that $R_0=+\infty$ if $\int_{\O} h(x)e^w\,dx=8\pi(1-\a)$ while $R_0<+\infty$ if 
$\int_{\O} h(x)e^w\,dx<8\pi(1-\a)$. 
On the other hand, in $\wtilde\O_-$ we will consider the annular, radial, decreasing, equimeasurable rearrangement $\phi^{**}$ of $\phi$. 
To this end, for any $t\in(t_-,c_0)$, $t_-=\min_\O \phi$, we let $R^-(t)$ be such that,
$$
	\int_{B_{R_0}\setminus B_{R^-(t)}}|x|^{-2\a}e^{U_{\a}}\,dx=\int_{\{\phi<t\}} h(x)e^w\,dx,
$$
Then, $\phi^{**}:B_{R_0}\setminus \ov B_{R^+_0}\to\R$, for $y\in B_{R_0}\setminus \ov B_{R^+_0}$, $|y|=r$, is given by,
$$
	\phi^{**}(r)=\inf\{t\in(t_-,c_0)\,:\,R^-(t)<r\}.
$$ 
We have,
\begin{align*}
	&\int_{B_{R_0}\setminus \ov B_{R^+_0}}|x|^{-2\a}e^{U_{\a}}\,dx=\int_{\wtilde\O_-} h(x)e^w\,dx, \\
	&\int_{B_{R_0}\setminus \ov B_{R^+_0}}|x|^{-2\a}e^{U_{\a}}|\phi^{**}|^2\,dx=\int_{\wtilde\O_-} h(x)e^w|\phi|^2\,dx, \\	
	&\int_{B_{R_0}\setminus \ov B_{R^+_0}}|x|^{-2\a}e^{U_{\a}}\phi^{**}\,dx=\int_{\wtilde\O_-} h(x)e^w\phi\,dx, \\
	&\int_{B_{R_0}\setminus \ov B_{R^+_0}} |\n \phi^{**}|^2\,dx \leq \int_\O |\n \phi|^2\,dx-\int_{\wtilde\O_+} |\n \phi|^2\,dx.
\end{align*}
Finally, we let $\phi_*:B_{R_0}\to\R$ be the following radial function,
$$
	\phi_*(r)= \left\{ \begin{array}{ll}
												\phi^*(r), & r\in[0,R^+_0], \vspace{0.2cm}\\
												\phi^{**}(r), & r\in(R^+_0,R_0).
										 \end{array} \right.
$$
By using the properties of $\phi^*, \phi^{**}$ and of $\phi$ {(see also \eqref{gg2})} we obtain,
\begin{align*}
	\int_{B_{R_0}} |\n \phi_*|^2\,dx & = \int_{B_{R^+_0}} |\n \phi^*|^2\,dx + \int_{B_{R_0}\setminus \ov B_{R^+_0}} |\n \phi^{**}|^2\,dx \\
				& \leq \int_{\wtilde\O_+} |\n \phi|^2\,dx + \int_\O |\n \phi|^2\,dx-\int_{\wtilde\O_+} |\n \phi|^2\,dx \\
				& = \int_\O |\n \phi|^2\,dx = \int_{\O} h(x)e^w|\phi|^2\,dx \\
				& = \int_{B_{R^+_0}}|x|^{-2\a}e^{U_{\a}}|\phi^{*}|^2\,dx+\int_{B_{R_0}\setminus \ov B_{R^+_0}}|x|^{-2\a}e^{U_{\a}}|\phi^{**}|^2\,dx \\
				& = \int_{B_{R_0}}|x|^{-2\a}e^{U_{\a}}|\phi_*|^2\,dx.
\end{align*}
Thus, we conclude that,
\begin{equation} \label{rearr-ineq3}
	\int_{B_{R_0}} |\n \phi_*|^2\,dx-\int_{B_{R_0}}|x|^{-2\a}e^{U_{\a}}|\phi_*|^2\,dx\leq0.
\end{equation}
Moreover, we have,
$$
	\int_{B_{R_0}}|x|^{-2\a}e^{U_{\a}}\phi_*\,dx=0.
$$
At this point observe that, defining  $K^*$ as in \eqref{K}, 
by construction and by the property \eqref{rearr-ineq3} of $\phi_*$, we obtain once more $K^*\leq 1$ for any $\rho\leq8\pi(1-\a)$.

\medskip

We distinguish now between two cases. Suppose first that $\rho<8\pi(1-\a)$. In this situation we follow step by step the argument   
in Case 1 (starting from \eqref{zero}) to conclude that $K^*>1$, which is a contradiction. Suppose now that $\rho=8\pi(1-\a)$ and argue once more as in 
Case 1 to conclude that $K^*=1$.
Therefore, the inequality in \eqref{rearr-ineq3} turns out to be an equality. This yields to the equality in \eqref{grad} as well and, 
in particular, to the equality in the Alexandrov-Bol inequality \eqref{bol-ineq}. We conclude that necessarily $\O_t$ and $\{\phi<t\}$ are simply-connected for a.e. $t>c_0$. 
But clearly $\{\phi<t\}$ is not simply-connected for $t<c_0$. This yields a contradiction and the proof is completed.\\
\end{proof}

\medskip

\section{The proof of Theorem \ref{thm}} \label{sec-proof}  

In this section we start by proving some new uniform estimates for solutions to \eqref{eq2} and then finally deduce the 
main result of Theorem \ref{thm} by making use of the spectral estimates introduced in the previous section.

\medskip

The following estimates are the counterpart of the well know results obtained in \cite{bm, bt, ls} and \cite{BM3} for the model case \eqref{eq}-\eqref{V} in the 
subcritical region. 
However it seems not easy to adapt those results to the general case in \eqref{eq2}. Thus we will adopt a different method 
based on rearrangement arguments in the same spirit of \cite{band,bc2}.
\begin{pro} \label{pro:est}
Let $\O_0\subset\R^2$ be a simple domain according to Definition \ref{def-simple}. Let $x_0\in\O_0$ be fixed as in Remark \ref{remxz} and let 
$w\in W^{2,s,\mbox{\rm \scriptsize loc}}(\O_0\setminus\{x_0\})\cap W^{2,q}(\O_0)\cap C^0(\ov\O_0)$ for some $s>2$ and some $q>1$, satisfy
$$
	\D w + h(x)e^{w}= 0  \mbox{ in } \O_0,
$$
where $h=e^H$ is such that $\a(\O_0,h)$ (as defined in \eqref{a}) satisfies $\a(\O_0,h)<1$.
Let $\omega\subseteq \O_0$ be any open subdomain such that $\p\o$ is a finite union of rectifiable Jordan curves and let $\a(\o)=\a(\o,h)$.  Suppose
$$
	M(\o)=\int_\o h(x)e^{w}\,dx<8\pi(1-\a(\o)).
$$
Then,
\begin{equation} \label{est-max}
	\max_{\ov\o}e^w\leq\left( 1-\frac{M(\o)}{8\pi(1-\a(\o))} \right)^{-2} \max_{\p\o}e^w.
\end{equation}
In particular, for any $\e\in(0,8\pi(1-\a(\O_0)))$, there exists $C_\e>0$ such that
\begin{equation} \label{unif}
	\|u_\rho\|_{L^{\infty}(\O_0)}\leq \rho \,C_\e,
\end{equation}
for any $\rho\in\bigr(0,8\pi(1-\a(\O_0))-\e\bigr)$ and for any solution $u_\rho\in W^{2,s,\mbox{\rm \scriptsize loc}}(\O_0\setminus\{x_0\})\cap W^{2,q}(\O_0)\cap C^0(\ov\O_0)$ of \eqref{eq2} with $\O=\O_0$. 
\end{pro}

\begin{proof}
Let $g$ and $\eta$ be the functions defined in the proof of Proposition \ref{bol} with $\eta=w-g$ and $\eta=0$ on $\p\o$, i.e. $g=w$ on $\p\o$. Let 
$$
	\O(t)=\{x\in\o \,:\, \eta(x)>t\}, \quad \mu(t)=\int_{\O(t)}h(x)e^g\,dx,
$$ 
and let $\eta^*$ be the rearrangement of $\eta$ as given in \eqref{eta*}. Finally, let $F(s)$ and $P(s)$ be defined as in the proof of Proposition \ref{bol} which we recall here for reader's convenience, 
$$
	F(s)=\int_{\O(\eta^*(s))} h(x)e^\eta e^g\,dx,
$$
$$
	P(s)=4\pi(1-\a(\o))\left(s F'(s)-F(s)\right)+\frac12 F^2(s).
$$
By \eqref{P} we have, for a.a. $s\in(0,\mu(0))$,
\begin{equation} \label{est-P}
	P(s)\geq P(0)=0.
\end{equation}
Let us introduce now, for a.a. $s\in(0,\mu(0))$,
$$
	J(s)=s\left( \frac{1}{F(s)}-\frac{1}{8\pi(1-\a(\o))} \right),
$$
and observe that $J(s)>0$ since by assumption, $F(s)\leq F(\mu(0))=M(\o)<8\pi(1-\a(\o))$. Moreover,
$$
	J'(s)=-\frac{P(s)}{4\pi(1-\a(\o))F^2(s)}\leq0
$$
by \eqref{est-P}. Thus, $J$ is non-increasing and in particular, for a.a. $s\in(0,\mu(0))$ we have, 
$$
	J(s)\leq\lim_{s\to0^+} J(s)
$$
which, by the l'Hopital theorem reads,
$$
	\lim_{s\to0^+} J(s) = \lim_{s\to0^+} \frac{s}{F(s)}= \lim_{s\to0^+}\frac{1}{F'(s)}= \frac{1}{F'(0)}\,.
$$ 
Recalling \eqref{F} we have $F'(0)=\max_{\ov\o}e^\eta$ and hence, for any a.a. $s\in(0,\mu(0))$ we obtain,
\begin{equation} \label{max}
	\max_{\ov\o}e^\eta \leq \frac{1}{J(s)}\,.
\end{equation}
We now estimate $J(s)$ by using \eqref{est-P} once more which is equivalent to, for a.a. $s\in(0,\mu(0))$,
$$
	sF'(s)\geq F^2(s)\left( \frac{1}{F(s)}-\frac{1}{8\pi(1-\a(\o))} \right).
$$
Recalling that $J(s)>0$ and $F'(s)>0$ we exploit the latter estimate to deduce, for a.a. $s\in(0,\mu(0))$,
\begin{align*}
	J(s) & = \frac{sF'(s)}{F'(s)}\left( \frac{1}{F(s)}-\frac{1}{8\pi(1-\a(\o))} \right) \\
			 & \geq \frac{F(s)}{F'(s)}\left( \frac{1}{F(s)}-\frac{1}{8\pi(1-\a(\o))} \right)^2 \\
			 & = \frac{1}{F'(s)}\left( 1-\frac{F(s)}{8\pi(1-\a(\o))} \right)^2 .
\end{align*}
Going back to \eqref{max} we obtain, for a.a. $s\in(0,\mu(0))$,
$$
	\max_{\ov\o}e^\eta \leq \frac{1}{F'(s)}\left( 1-\frac{F(s)}{8\pi(1-\a(\o))} \right)^{-2} .
$$
Finally, letting $s\to\mu(0)^-$ we eventually get, 
$$
	\max_{\ov\o}e^\eta \leq \left( 1-\frac{M(\o)}{8\pi(1-\a(\o))} \right)^{-2} .
$$
On the other hand, recall that $\eta=w-g$ with $g$ harmonic and $g=w$ on $\p\o$. Thus, by making use also of the weak maximum principle we obtain,
\begin{align*}
	\max_{\ov\o} e^w & =\max_{\ov\o} e^{\eta+g} \leq \left( 1-\frac{M(\o)}{8\pi(1-\a(\o))} \right)^{-2} \max_{\ov\o} e^g \\
				& = \left( 1-\frac{M(\o)}{8\pi(1-\a(\o))} \right)^{-2} \max_{\p\o} e^g \\
				& = \left( 1-\frac{M(\o)}{8\pi(1-\a(\o))} \right)^{-2} \max_{\p\o} e^w,
\end{align*}
which is the desired estimate in \eqref{est-max}.
\end{proof}

\medskip

We can now prove the main theorem.

\begin{proof}[Proof of Theorem \ref{thm}.]
We first aim to show that the linearized equation for \eqref{eq2} has strictly positive first eigenvalue for any $\rho\leq 8\pi(1-\a)$, where $\a=\a(\O,h)<1$ is given in \eqref{a}. Indeed, suppose by contradiction that there exist a solution $u$ for \eqref{eq2} with $\rho\leq 8\pi(1-\a)$ and a non-trivial eigenfunction $\tilde\phi\in H^1_0(\O)$, that is, 
$$
\left\{ \begin{array}{ll}
 -\D\tilde\phi-\rho\dfrac{h(x)e^u}{\int_\O h(x)e^u\,dx}\left( \tilde\phi-\dfrac{\int_\O h(x)e^u\tilde\phi\,dx}{\int_\O h(x)e^u\,dx} \right)=0 & \mbox{in } \O, \vspace{0.2cm}\\
 \tilde\phi=0 & \mbox{on } \p \O.
\end{array}
\right.
$$
Letting 
$$
	w=u+\log\rho-\log\left(\int_\O h(x)e^u\,dx\right), \quad \phi=\tilde\phi-\dfrac{\int_\O h(x)e^u\tilde\phi\,dx}{\int_\O h(x)e^u\,dx}\,,
$$
we have,
$$
\left\{ \begin{array}{ll}
-\D\phi -h(x)e^w\phi=0 & \mbox{in } \O, \vspace{0.2cm}\\
 \phi=c_0 & \mbox{on } \p \O, \vspace{0.2cm}\\
\int_\O h(x)e^w\phi\,dx=0,
\end{array}
\right.
$$
for some $c_0\in\R$. Without loss of generality we may assume $c_0\leq0$. The goal is to show that $\phi\equiv c_0$ which in turn implies $\tilde\phi\equiv 0$ yielding a contradiction.

\medskip

Observe that $\int_\O h(x)e^w\phi\,dx=0$. 
Therefore, if $c_0=0$, $\phi$ must change sign unless $\phi\equiv 0$. 
{Then, by Lemma \ref{lem-nodal}, $\phi$ is an eigenfunction for \eqref{lin1} corresponding to an eigenvalue $\hat\nu_k$ for some $k\geq2$.} 
On the other hand, Proposition \ref{pro-nodal} $(ii)$ implies $\hat\nu_2>0$. Then, we necessarily have $\phi\equiv 0$ as claimed. 
We are left with the case $c_0<0$. But in this case we may apply Proposition \ref{pro-nodal2} which asserts that $\hat\nu>0$ which is a contradiction.

\medskip

Therefore, we conclude that the linearized operator $L_\rho$ for \eqref{eq2} has strictly positive first eigenvalue for any $\rho\leq8\pi(1-\a)$. 
Let now $S_\rho$ be the branch of solutions for \eqref{eq2} bifurcating from $(u,\rho)=(0,0)$. 
By standard bifurcation theory \cite{cr-rab}, we deduce that $S_\rho$ is a simple branch near $\rho=0$. 
In particular, for any $\rho>0$ small enough, there exists a unique solution for \eqref{eq2}. 
Moreover, since $L_\rho$ has strictly positive first eigenvalue we can apply the implicit function theorem to extend 
uniquely $S_\rho$ for any $\rho<8\pi(1-\a)$. Suppose by contradiction there exists a second (non-bending) branch of solutions for $\rho<8\pi(1-\a)$. 
Then, the estimates in \eqref{unif} implies that the latter branch intersects $S_\rho$ in $(u,\rho)=(0,0)$, which can not happen. 
We conclude uniqueness of solutions for \eqref{eq2} holds for $\rho<8\pi(1-\a)$. 
Finally, assume by contradiction that there exist more than on solution for $\rho=8\pi(1-\a)$. Since we can apply the implicit function theorem around each one of 
these solutions, then we readily obtain a contradiction to the uniqueness for $\rho<8\pi(1-\a)$.
\end{proof}

\

\end{document}